\documentclass[oneside,10pt]{article}          
\usepackage[b5paper]{geometry}	    
\usepackage{amsfonts,amsmath,latexsym,amssymb} 
\usepackage{amsthm}                
\usepackage{mathrsfs,upref}         
\usepackage{mathptmx}		    
\usepackage{jca}	            

\usepackage{enumerate}
\usepackage{enumitem}
\usepackage{scalerel}[2014/03/10]
\usepackage[usestackEOL]{stackengine}
\usepackage{hyperref}
\usepackage{xcolor}
\hypersetup{
  colorlinks   = true, 
  urlcolor     = blue, 
  linkcolor    = blue, 
  citecolor   = red 
}
\usepackage{etoolbox}
\makeatletter
\let\sv@thm\@thm
\def\@thm{\let\indent\relax\sv@thm}
\makeatother

\newtheorem{lemma}{Lemma}               
\newtheorem{corollary}{Corollary}

\newtheorem{theorem}{Theorem}[section]         

\theoremstyle{definition}

\newtheorem{remark}{Remark}[section]

\numberwithin{equation}{section}       

\DeclareMathOperator{\polylog}{Li}
\DeclareMathOperator{\arctanh}{arctanh}
\DeclareMathOperator{\Ein}{Ein}
\DeclareMathOperator{\barnesG}{\textbf{G}}

\def\dashint{\,\ThisStyle{\ensurestackMath{%
  \stackinset{c}{.2\LMpt}{c}{.5\LMpt}{\SavedStyle-}{\SavedStyle\phantom{\int}}}%
  \setbox0=\hbox{$\SavedStyle\int\,$}\kern-\wd0}\int}

\begin{document}

\title[Short title]{On Some Series Involving Harmonic and Skew-Harmonic Numbers}


\author{Vincent Nguyen}

\address{Vincent Nguyen\\
Independent student researcher\\
28000 Marguerite Pkwy, Mission Viejo, CA 92692\\
\email{21nguyenvh@gmail.com}}

\dedicated{Dedicated to Minh}                    

\date{DD.MM.2023}                               




\begin{abstract}
        In this paper, we evaluate in closed form several different series involving the harmonic numbers and skew-harmonic numbers. We consider two classes of series involving these sequences. One class of series involves the product of the $n$th harmonic or skew-harmonic number and a tail. We provide the solution to two open problems concerning these harmonic series with tails from Alina Sîntămărian and Ovidiu Furdui's book Sharpening Mathematical Analysis Skills. The other class of series is the Hardy series, which involves a logarithm and the Euler-Mascheroni constant being subtracted from the $n$th harmonic number.
\end{abstract}

\maketitle



\section{Introduction}
The harmonic numbers are defined as $H_n = 1 + \frac{1}{2} + \ldots + \frac{1}{n}$. Its generalization is defined as $H_n^{(p)} = 1 + \frac{1}{2^p} + \ldots + \frac{1}{n^p}$. The generalized harmonic numbers are useful in evaluating some of our series in closed form. The alternating analogue of the harmonic numbers, skew-harmonic numbers, are defined to be $\overline{H}_n = 1 - \frac{1}{2} + \ldots + \frac{(-1)^{n-1}}{n}$. Leonhard Euler famously derived
\begin{equation*}
    \sum_{n = 1}^\infty \frac{H_n}{n^k} = \left(1 + \frac{k}{2}\right)\zeta(k+1) - \frac{1}{2} \sum_{n = 1}^{k-2} \zeta(k-n)\zeta(n + 1) \quad k \in \{2, 3, \ldots\},
\end{equation*}
where $\zeta(s) = \sum_{n = 1}^\infty n^{-s}$ is the Riemann zeta function (see \cite[p. 807]{AS}, \cite[25.2]{DLMF}). The result can be found in \cite{BorweinEuler} and proved in \cite[pp. 47-49]{NielsEuler}. Other formulae for Euler sums have been derived, one of which we will use from ~\cite[Theorem 7.2, p. 33]{EulerSumContour}:
\begin{align}
\label{AltGenEulerSum}
    \begin{split}
        \sum_{n=1}^\infty (-1)^{n-1} \frac{H_n^{(p)}}{n^q}
        =& \frac{(1-(-1)^p)\zeta(p) \eta(q) + \eta(p+q)}{2} \\ 
        &+ \sum_{j + 2k = p} \binom{q + j -1}{q - 1} (-1)^{j+1}\eta(q + j) \eta(2k)\\
        &+ (-1)^p \sum_{i + 2k = q} \binom{p + i -1}{p - 1}\zeta(p + i)\eta(2k), \quad p+q \text{ is odd}
    \end{split}
\end{align}
where $\eta(s) = \sum_{n = 1}^\infty \frac{(-1)^{n-1}}{n^s}$ denotes the Dirichlet eta function (see \cite[p. 807]{AS}).

Series involving the Riemann zeta function can be found in \cite{MiscSeriesZeta} and \cite{SrivastavaEuler}. Other infinite series evaluated in closed form incorporating the Riemann zeta function and harmonic numbers can be found in \cite{BaileyEuler} and \cite{BorweinEuler}. Special values of the zeta function can be found in \cite[p. 807]{AS}.

In this paper, we will be evaluating in closed form two classes of series involving harmonic and skew-harmonic numbers. One of such is a series with a tail. Sîntămărian and Furdui derive in \cite[Problem 7.90 (a)]{MathAnalysisBook} the following harmonic series with tail of $\zeta(2)$:
\begin{equation*}
    \sum_{n = 1}^\infty H_n\left(\zeta(2) - 1 - \frac{1}{2^2} - \ldots - \frac{1}{n^2} - \frac{1}{n} \right) = -1
\end{equation*}
and pose the following related open problem in \cite[p. 214]{MathAnalysisBook}, which we will calculate in this paper:
\begin{equation*}
    S_1 := \sum_{n = 1}^\infty H_{2n}\left(\zeta(2) - 1 - \frac{1}{2^2} - \ldots - \frac{1}{n^2} - \frac{1}{n} \right).
\end{equation*}
Sîntămărian and Furdui derive in \cite[Problem 3.81 (c)]{MathAnalysisBook} another harmonic series with a tail of $\zeta(2)$:
\begin{align*}
    \sum_{n = 1}^\infty \frac{\overline{H}_n}{n}\left(\zeta(2) - 1 - \frac{1}{2^2} - \ldots - \frac{1}{n^2} \right) 
    =& \frac{53\pi^4}{1440} + \frac{\pi^2 \log^2(2)}{4} - \frac{\log^4(2)}{8}\\
    &- \frac{21 \log(2)}{8} \zeta(3) -3 \polylog_4\left(\frac{1}{2}\right)
\end{align*}
where $\polylog_s(z) = \sum_{n = 1}^\infty \frac{z^n}{n^s}$ denotes the polylogarithm of order $s$ (see ~\cite[p. 189]{PolyLogMisc}, ~\cite[25.12]{DLMF}). Note that $\polylog_s(1) = \zeta(s)$ for $s > 1$. Sîntămărian and Furdui also pose in \cite[p. 101]{MathAnalysisBook} a similar harmonic series as an open problem with a tail of $\zeta(3)$,
\begin{equation*}
    S_2 := \sum_{n = 1}^\infty \frac{\overline{H}_n}{n}\left(\zeta(3) - 1 - \frac{1}{2^3} - \ldots - \frac{1}{n^3} \right),
\end{equation*}
which we will calculate in this paper.

The other class of series we will evaluate in closed form is the Hardy series, which can be thought of as a series involving the Euler-Mascheroni constant $\gamma$ (see ~\cite[pp. 28-34]{FinchConstant}) and the logarithm subtracted from the $n$th harmonic number. Hardy derived the alternating Hardy series (see ~\cite[p. 277]{SeriesBook})
$$\sum_{n = 1}^\infty (-1)^{n}\left(H_n - \log(n) - \gamma\right) = \frac{\gamma - \log(\pi)}{2}.$$
We will provide a generalization of this sum as well as analogues of this generalization in section \ref{HardySeriesSect}.

To evaluate these series, we require Abel's summation formula (see ~\cite[p. 38]{MathAnalysisBook}), which states that
\begin{equation}
    \sum_{k = 1}^\infty a_k b_k = \lim_{n \to \infty }A_n b_{n+1} + \sum_{k = 1}^\infty A_k\left(b_k - b_{k+1}\right) \label{AbelSum}
\end{equation}
if $\left(a_n\right)_{\geq 1}$ and $\left(b_n\right)_{\geq 1}$ are two sequences of real numbers and $A_n := \sum_{k = 1}^n a_n$.

\section{Harmonic Series with a Tail}
We will now evaluate some harmonic series with tails. In order to do so, we will make use of the integral representations of the harmonic and skew-harmonic numbers. Due to Euler (see ~\cite{SandiferEuler}), we have
\begin{equation} \label{HarmonicInt}
    H_n = \int_0^1 \frac{1 - x^n}{1-x}\,dx, \quad n \in \{1, 2, \ldots\}.
\end{equation}

\noindent
Similarly, we also have the following integral representation for the skew-harmonic numbers (see ~\cite[Problem 3.66, p. 95]{MathAnalysisBook}):
\begin{equation} \label{SkewHarmonicInt}
    \overline{H}_n = \int_0^1 \frac{1 - (-x)^n}{1+x}\,dx, \quad n \in \{1, 2, \ldots\}.
\end{equation}

\subsection{The Calculation of \texorpdfstring{$S_1$}{S1}}
\noindent
To evaluate $S_1$, we will establish some lemmas.

\begin{lemma} \label{AtanhIntLemma}
    The following equality holds:
    $$\int_0^1 \frac{\arctanh(x) \log\left(1-x^2\right)}{x}\,dx = -\frac{7}{8}\zeta(3)$$
    where $\arctanh(x) = \frac{1}{2}\log\left(\frac{1+x}{1-x}\right)$ is the inverse hyperbolic tangent function for $x < 1$.
\end{lemma}

\begin{proof}
    We have
    \begin{align}
        \int_0^1 \frac{\arctanh(x) \log\left(1-x^2\right)}{x}\,dx
        &= \frac{1}{2}\int_0^1 \frac{\log\left(\frac{1+x}{1-x}\right) \log\left(1 - x^2\right)}{x}\,dx \nonumber \\
        &= \frac{1}{2}\int_0^1 \frac{\log^2(1 + x)}{x}\,dx - \frac{1}{2}\int_0^1 \frac{\log^2(1 - x)}{x}\,dx \label{atanhfinalstep}.
    \end{align}

    \noindent
    From \cite[p. 134]{HarmonicSeriesLogIntBook}, we make use of the following identity:
    \begin{align}
        \label{GeneralLogPower+}
        \begin{split}
            \int_0^1 \frac{\log^q(1+x)}{x}\,dx 
            =& \frac{\log^{q+1}(2)}{q+1} +q! \zeta(q+1)\\
            &- q!\sum_{k = 0}^q \frac{\log^{q - k}(2)}{(q - k)!} \polylog_{k + 1}\left(\frac{1}{2}\right), \quad q \in \{1, 2, \ldots\}.
        \end{split}
    \end{align}

    \noindent
    We get that
    \begin{equation} \label{AtanhIntp1}
        \int_0^1 \frac{\log^2(1 + x)}{x}\,dx = \frac{\zeta(3)}{4}
    \end{equation}
    using \eqref{GeneralLogPower+} with $q =2$ and plugging in the following special values of the polylogarithms (see ~\cite[pp. 6, 155]{PolyLogMisc}:
    $$\polylog_1\left(\frac{1}{2}\right) = \log(2), \quad \polylog_2\left(\frac{1}{2}\right) = \frac{\pi^2}{12} - \frac{\log^2(2)}{2},$$
    $$\polylog_3\left(\frac{1}{2}\right) = \frac{7}{8}\zeta (3) + \frac{\log^3(2)}{6}-\frac{\pi^2}{12}\log(2).$$

    \noindent
    To deal with the other integral, we simply integrate by substituting $x \mapsto 1 - e^{-x}$ and make use of the identity (see \cite[p. 189]{Conway})
    $$\int_0^\infty \frac{t^{s-1}}{e^t-1}\,dt = \zeta(s)\Gamma(s)$$
    where $\Gamma(s) = \int_0^\infty t^{s-1} e^{-t}\,dt$ is the classical gamma function (see ~\cite[5.2, 5.4]{DLMF}, ~\cite[pp. 255-257]{AS}) to get that
    \begin{equation} \label{AtanhIntp2}
        \int_0^1 \frac{\log^2(1 - x)}{x}\,dx = 2\zeta(3).
    \end{equation}

    \noindent
    Plugging in \eqref{AtanhIntp1} and \eqref{AtanhIntp2} into \eqref{atanhfinalstep}, we get our desired result.
    
\end{proof}

\noindent
We note that the special value $\polylog_1(1/2) = \log(2)$ is obvious using the power series representation of the polylogarithm.

\begin{lemma} \label{loglemma}
    The following equality holds:
    $$\int_0^1 \frac{\log\left( 1 - x^2\right)}{x}\,dx = -\frac{\pi^2}{12}.$$
\end{lemma}

\begin{proof}
    Integrate by substituting $x \mapsto \sqrt{x}$. From there, we can make use of the identities $-\int_0^z \frac{\log(1-x)}{x} = \polylog_2(z)$ and $\polylog_2(1) = \zeta(2) = \frac{\pi^2}{6}$ (see \cite[pp. 1, 4]{PolyLogMisc}).
\end{proof}

\noindent
We can now begin evaluating $S_1$ in closed form.

\noindent
\begin{theorem}
    The following equality holds:
    $$\sum_{n = 1}^\infty H_{2n}\left(\zeta(2) - 1 - \frac{1}{2^2} - \ldots - \frac{1}{n^2} - \frac{1}{n} \right) = \log(2) - \frac{7}{8}\zeta(3) - 1.$$
\end{theorem}

\begin{proof}
    We begin by plugging \eqref{HarmonicInt} into $S_1$:
    \begin{align*}
        S_1 &= \sum_{n=1}^\infty \int_0^1 \frac{1-x^{2n}}{1-x}\left(\zeta(2) - 1 - \frac{1}{2^2} - \ldots - \frac{1}{n^2} - \frac{1}{n}\right)\,dx\\
        &= -\sum_{n=1}^\infty \int_0^1 \frac{1 - x^{2n}}{1-x}\left(\frac{1}{n} +  H_n^{(2)} - \zeta(2)\right)\,dx
    \end{align*}

    \noindent
    Since $\frac{1-x^{2n}}{1-x} \geq 0$ and $\frac{1}{n} + H_n^{(2)} - \zeta(2) \geq 0$\footnote{To verify this, let $c_n := \frac{1}{n} + H_n^{(2)} - \zeta(2)$. Notice that $c_{n+1} - c_n = -\frac{1}{n(n+1)^2} < 0$, $c_1 = 2 - \zeta(2) > 0$ and $\lim_{n \to \infty} c_n = 0$. Thus, $c_n \geq 0$ for integers $n \geq 1$.}
    for integers $n \geq 1$ and $x \in [0,1)$, we may apply Tonelli's theorem for non-negative functions (see ~\cite[Theorem 23.17, p. 558]{MeasureTheoryBook}) and interchange the order of summation and integration. Making use of the identity (see ~\cite[Problem 2.47, p. 39]{MathAnalysisBook})
    $$\sum_{n = 1}^\infty \left(\zeta(2) - 1 - \frac{1}{2^2} - \ldots - \frac{1}{n^2} - \frac{1}{n+k}\right) = H_{k+1} + \frac{k}{k+1} - \zeta(2), \quad k \in\{0, 1, 2,\ldots\}$$
    with $k = 0$, the generating function from \cite[p. 65]{MathAnalysisBook}
    \begin{equation*} 
    \sum_{n = 1}^\infty \left(\zeta(2) - 1 - \frac{1}{2^2} - \ldots - \frac{1}{n^2} \right)x^n = \dfrac{x \zeta(2) - \polylog_2(x)}{1 - x}, \quad x \in [-1,1)
    \end{equation*}
    and the known power series for the function $-\log(1-x)$, we get
    \begin{align*}
        S_1
        = &\int_0^1 \frac{1}{1 - x } \left(1 - \frac{\zeta(2) - \polylog_2\left(x^2\right)}{1 - x^2} - \log(1-x) - \log(1+x)\right)\,dx\\
        \begin{split}
            = &\lim_{t \to 1^-} \Bigg[\frac{\log^2(1-t)}{2} -\log(1-t)  - \frac{\zeta(2)}{2}\left( \arctanh(t) + \frac{t}{1-t} \right)\\
            &+ \int_0^t \frac{\polylog_2\left(x^2\right)}{(1-x)^2(1+x)}\,dx - \int_0^t \frac{\log(1+x)}{1-x}\,dx\Bigg].
        \end{split}\stepcounter{equation}\tag{\theequation}\label{Thm1Step}
    \end{align*}

    \noindent
    We now have two integrals
    \begin{equation} \label{logdilogstep}
        I_1 := \int_0^t \frac{\polylog_2\left(x^2\right)}{(1-x)^2(1+x)}\,dx - \int_0^t \frac{\log(1+x)}{1-x}\,dx.
    \end{equation}

    \noindent
    We can apply apply integration by parts with the derivative of the dilogarithm (see ~\cite[p. 1]{PolyLogMisc}):
    \begin{align*}
        \int_0^t \frac{\polylog_2\left(x^2\right)}{(1-x)^2(1+x)}\,dx
        =& \frac{1}{2}\int_0^t \frac{\polylog_2 \left(x^2\right)}{(1-x)^2}\,dx + \frac{1}{2}\int_0^t \frac{\polylog_2\left(x^2\right)}{1 - x^2}\,dx\\
        \stepcounter{equation}\tag{\theequation}\label{dilogstep1}
        \begin{split}
            =& \frac{\polylog_2\left(t^2\right)}{2(1 - t)} + \int_0^t \left(\frac{\log\left(1 - x^2\right)}{x} + \frac{\log\left(1 - x^2\right)}{1 - x}\right)\,dx \\
            &+ \frac{\arctanh(t)\polylog_2\left(t^2\right)}{2} +\int_0^t \frac{\arctanh(x) \log\left(1 - x^2\right)}{x}\,dx.
        \end{split}
    \end{align*}

    \noindent
    Substituting \eqref{dilogstep1} into \eqref{logdilogstep}, we get
    \begin{align*}
        I_1
        =& \frac{\polylog_2\left(t^2\right)}{2(1 - t)} + \int_0^t \frac{\log\left(1 - x^2\right)}{x}\,dx + \int_0^t \frac{\log(1-x)}{1 - x}\,dx \\
        &+ \frac{\arctanh(t)\polylog_2\left(t^2\right)}{2} +\int_0^t \frac{\arctanh(x) \log\left(1 - x^2\right)}{x}\,dx\\
        \stepcounter{equation}\tag{\theequation}\label{dilogstep2}
        \begin{split}
            =& \frac{\polylog_2\left(t^2\right)}{2(1 - t)} - \frac{\log^2(1 -t)}{2} + \int_0^t \frac{\log\left(1 - x^2\right)}{x}\,dx + \frac{\arctanh(t)\polylog_2\left(t^2\right)}{2}\\
            & +\int_0^t \frac{\arctanh(x) \log\left(1 - x^2\right)}{x}\,dx.
        \end{split}
    \end{align*}

    \noindent
    Plug \eqref{dilogstep2} into \eqref{Thm1Step}:
    \begin{align*}
        S_1
        = & \lim_{t \to 1^-} \Bigg[\frac{\arctanh(t)}{2}\left(\polylog_2\left(t^2\right) - \zeta(2)\right)-\log(1-t) + \frac{\polylog_2\left(t^2\right) -t \zeta(2) }{2(1-t)}\\
        & + \int_0^t \frac{\log\left(1 - x^2\right)}{x}\,dx +\int_0^t \frac{\arctanh(x) \log\left(1 - x^2\right)}{x}\,dx\Bigg].
    \end{align*}

    \noindent
    By Lemmas \ref{AtanhIntLemma} and \ref{loglemma}, we have
    \begin{align}
        \label{finalsubThm1}
        \begin{split}
            S_1 = & - \frac{\pi^2}{12} - \frac{7}{8}\zeta(3) + \lim_{t \to 1^-} \frac{\arctanh(t)}{2}\left(\polylog_2\left(t^2\right) - \zeta(2)\right)\\
            &+ \lim_{t \to 1^-} \frac{\polylog_2\left(t^2\right) -t \zeta(2) -2(1-t)\log(1-t)}{2(1-t)}.
        \end{split}
    \end{align}

    \noindent
    By applying L'Hôpital's rule and making use of the fact that $\zeta(2) = \frac{\pi^2}{6}$ (see ~\cite[p. 807]{AS}), we obtain
    \begin{equation} \label{limit1Thm1}
        \lim_{t \to 1^-} \frac{\polylog_2\left(t^2\right) -t \zeta(2) -2(1-t)\log(1-t)}{2(1-t)} = \frac{\pi^2}{12} - 1 + \log (2).
    \end{equation}

    \noindent
    It is known that the dilogarithm satisfies the functional equation (see ~\cite[p. 5]{PolyLogMisc}):
    $$\frac{\pi^2}{6} - \polylog_2(x) = \polylog_2(1-x) + \log(x) \log(1-x).$$
    Using this identity and the special value $\polylog_2(1) = \zeta(2) = \frac{\pi^2}{6}$, we get that
    \begin{equation} \label{limit2Thm1}
        \lim_{t \to 1^-} \frac{\arctanh(t)}{2}\left(\polylog_2\left(t^2\right) - \zeta(2)\right) = 0.
    \end{equation}
    
    \noindent
    Plugging \eqref{limit1Thm1} and \eqref{limit2Thm1} into \eqref{finalsubThm1} and simplifying gives us the desired result.
\end{proof}

\subsection{The Calculation of \texorpdfstring{$S_2$}{S2}}
In order to evaluate in closed form $S_2$, we establish some lemmas.

\begin{lemma} \label{Thm2Int1Lemma}
    The following equality holds:
    $$\int_0^1 \frac{\polylog_4(-x)}{1 + x}\,dx = \frac{17}{16}\zeta (5) -\frac{3}{8}\zeta(2)\zeta(3) -\frac{7}{8}\log(2)\zeta(4).$$
\end{lemma}

\begin{proof}
    Begin by integrating by parts:
    \begin{align*}
        \int_0^1 \frac{\polylog_4(-x)}{1+x}\,dx
        &= \log(2) \polylog_4(-1) - \int_0^1 \frac{\log(1 + x) \polylog_3(-x)}{x}\,dx\\
        &= -\frac{7}{8}\log(2)\zeta(4) - \int_0^1 \frac{\log(1 + x) \polylog_3(-x)}{x}\,dx.
    \end{align*}
    We differentiated the polylogarithm using the identity $\polylog_n(z) = \int_0^z \frac{\polylog_{n-1}(t)}{t}\,dt$ (see \cite[p. 189]{PolyLogMisc}). We also used the identities $\polylog_n(-1) = -\eta(n)$, which can be verified via the power series in \cite[p. 189]{PolyLogMisc}, and $\eta(s) = \left(1 - 2^{1 - s}\right)\zeta(s)$ (see ~\cite[23.2.19, p. 807]{AS}).
    With the Cauchy product of $- \log(1 - x)\polylog_3(x)$ (see ~\cite[2.4.4, p. 101]{HarmonicSeriesLogIntBook}\footnote{The author of second edition of this book has made some typos in the proof of the Cauchy product.})
    $$- \log(1 - x) \polylog_3(x) = 2 \sum_{n = 1}^\infty \frac{H_n}{n^3} x^n + \sum_{n = 1}^\infty \frac{H_n^{(2)}}{n^2}x^n + \sum_{n = 1}^\infty \frac{H_n^{(3)}}{n}x^n - 4 \polylog_4(x),$$
    we may switch the order of summation and integration \footnote{Abel's Theorem (see ~\cite[Theorem 45, p. 163]{BomanRogersAnalysis}) tells us that the power series converges uniformly on $x \in [0,1]$, thus justifying the interchange of the sum and integral.}. Integrating gives us
    \begin{align*}
    - \frac{7}{8}\log(2)\zeta(4) + 2 \sum_{n = 1}^\infty (-1)^n \frac{H_n}{n^4} + \sum_{n = 1}^\infty (-1)^n \frac{H_n^{(2)}}{n^3} + \sum_{n = 1}^\infty (-1)^n\frac{H_n^{(3)}}{n^2} - 4 \polylog_5(-1).
    \end{align*}

    \noindent
    Integration of the polylogarithm was carried out using \cite[7.2, p. 189]{PolyLogMisc}. We can then evaluate the remaining series using formula \eqref{AltGenEulerSum} with the appropriate $p$ and $q$ values. We once again make use of the polylogarithm's relationship to the Dirichlet eta function in terms of the Riemann zeta function using \cite[23.2.19, p. 807]{AS}. The values of these series are also given in \cite[pp. 310, 311]{AlmostImpossibleBook1}. Simplifying leaves us with the desired result.
    
\end{proof}

\noindent
We made use of the identity $\polylog_s(-1) = -\eta(s)$.

The following lemmas will require the use of the following result due to Vălean (see ~\cite[Lemma 6, p. 4]{ValeanSkewHarmonicPaper}):
\begin{align} \label{ValeanLogInt}
   \begin{split}
       \int_0^1 \frac{\log^2(1 + x) \polylog_2(-x)}{x}\,dx
       = &\frac{2}{15}\log^5(2) -\frac{2}{3}\log^3(2)\zeta(2) +\frac{7}{4}\log^2(2)\zeta (3)\\
       & -\frac{1}{8}\zeta(2)\zeta(3) -\frac{125}{32}\zeta(5) +4 \log(2)\polylog_4\left(\frac{1}{2}\right)\\
       &+  4\polylog_5\left(\frac{1}{2}\right).
   \end{split}
\end{align}

\begin{lemma} \label{Thm2Int2Lemma}
    The following equality holds:
    \begin{align*}
        \int_0^1 \frac{\log(1 + x)\polylog_3(-x)}{1 + x}\,dx
        =& \frac{1}{16}\zeta(2)\zeta (3)+\frac{125}{64}\zeta (5) -\frac{\log^5(2)}{15} +\frac{1}{3}\log^3(2)\zeta(2)\\
        &-\frac{5}{4}\log^2(2)\zeta(3) -2 \log(2) \polylog_4\left(\frac{1}{2}\right) -2 \polylog_5\left(\frac{1}{2}\right).
    \end{align*}
\end{lemma}

\begin{proof}
    We begin by integrating by parts:
    \begin{align*}
        \int_0^1 \frac{\log(1 + x)\polylog_3(-x)}{1 + x}\,dx
        &= \frac{\log^2(2) \polylog_3(-1)}{2} - \frac{1}{2}\int_0^1 \frac{\log^2(1+x) \polylog_2(-x)}{x} \,dx\\
        &= - \frac{3}{8}\log^2(2) \zeta(3) - \frac{1}{2}\int_0^1 \frac{\log^2(1+x) \polylog_2(-x)}{x}\,dx.
    \end{align*}
    We immediately obtain the desired result with \eqref{ValeanLogInt}.
\end{proof}

\begin{lemma} \label{Thm2Int3Lemma}
    The following equality holds:
    \begin{align*}
        \int_0^1 \frac{\polylog_2^2(-x)}{1 + x}\,dx 
        =& \frac{4}{15}\log^5(2) -\frac{4}{3}\log^3(2)\zeta(2) +\frac{7}{2} \log^2(2) \zeta(3) +\frac{5}{8}\log(2)\zeta(4)\\
        &-\frac{125}{16}\zeta(5) -\frac{1}{4}\zeta(2)\zeta(3) + 8\log(2)\polylog_4\left(\frac{1}{2}\right) + 8 \polylog_5\left(\frac{1}{2}\right).
    \end{align*}
\end{lemma}

\begin{proof}
    We begin by integrating by parts:
    \begin{align*}
        \int_0^1 \frac{\polylog_2^2(-x)}{1 + x}\,dx
        &= \log(2) \polylog_2^2(-1) + 2\int_0^1 \frac{\log^2(1+x) \polylog_2(-x)}{x}\,dx\\
        &= \log(2) \cdot \frac{\zeta^2(2)}{4} + 2\int_0^1 \frac{\log^2(1+x) \polylog_2(-x)}{x}\,dx.
    \end{align*}
    Making use of the fact that $\zeta^2(2) = \frac{5}{2}\zeta(4)$, since $\zeta(2) = \frac{\pi^2}{6}$ and $\zeta(4) = \frac{\pi^4}{90}$ (see \cite[p. 807]{AS}), and plugging in \eqref{ValeanLogInt}, we get the desired result.
\end{proof}

We can now begin to evaluate $S_2$ in closed form.

\begin{theorem}
    The following equality holds:
    \begin{align*}
        \sum_{n = 1}^\infty \frac{\overline{H}_n}{n}\left(\zeta(3) - 1 - \frac{1}{2^3} - \ldots - \frac{1}{n^3} \right)
        =& \frac{193}{64}\zeta(5) - \frac{5}{16}\zeta(2)\zeta(3) - \frac{\log^5(2)}{15}\\
        &+ \frac{1}{3}\log^3(2)\zeta(2) - \frac{15}{16}\log(2)\zeta(4)\\
        &-2\log(2)\polylog_4\left(\frac{1}{2}\right) - 2\polylog_5\left(\frac{1}{2}\right).
    \end{align*}
\end{theorem}

\begin{proof}
    Plugging in the integral representation of the skew-harmonic number \eqref{SkewHarmonicInt}, we have
    \begin{align*}
        S_2
        &= \sum_{n =1}^\infty \frac{1}{n}\int_0^1 \frac{1 - (-x)^n}{1+x}\left(\zeta(3) - 1 - \frac{1}{2^3} - \ldots - \frac{1}{n^3}\right)\,dx.
    \end{align*}

    \noindent
    Since $\frac{1 -(-x)^n}{1+x} \geq 0$ and $\zeta(3) - H_n^{(3)} >0$ for integers $n \geq 1$ and real numbers $0 \leq x \leq 1$, we may use Tonelli's theorem for non-negative functions (see ~\cite[Theorem 23.17, p. 558]{MeasureTheoryBook}) to justify interchanging the order of summation and integration. We make use of the generating function
    \begin{align*}
    &\sum_{n = 1}^\infty \left(\zeta(3) - 1 - \frac{1}{2^3} - \ldots - \frac{1}{n^3}\right) \frac{x^n}{n}\\
    =&
    \begin{cases}
        \log(1-x)\left[\polylog_3(x) - \zeta(3) \right] - \polylog_4(x) + \frac{\polylog_2^2(x)}{2} & \text{if } x \in [-1,1)\\
        \frac{\zeta(4)}{4} & \text{if } x = 1
    \end{cases}
    \end{align*}
    (see \cite[Problem 7.87 (b), p. 212]{MathAnalysisBook}), leaving us with
    \begin{align*}
        S_2
        = & \frac{\zeta(4)}{4}\log(2) - \int_0^1 \frac{\log(1+x)\polylog_3(-x)}{1+x}\,dx +\zeta(3)\int_0^1 \frac{\log(1 + x)}{1 + x}\,dx\\
        &+ \int_0^1 \frac{\polylog_4(-x)}{1+x}\,dx - \frac{1}{2}\int_0^1 \frac{\polylog_2^2(-x)}{1 + x}\,dx.
    \end{align*}
    Using $\int_0^1 \frac{\log(1+x)}{1+x}\,dx = \frac{1}{2}\log^2(2)$, plugging in Lemmas \ref{Thm2Int1Lemma}, \ref{Thm2Int2Lemma} and \ref{Thm2Int3Lemma}, and simplifying gives us the desired result.

\end{proof}

\subsection{Harmonic Series with Tail \texorpdfstring{$e^y$}{exp}}
We will derive a harmonic series with tail $e^y$, similar to the following series (see ~\cite[p. 10]{ExoticSeries}):
$$\sum_{n = 1}^\infty H_n\left(e^y - 1 - \frac{y}{1!} - \frac{y^2}{2!} - \ldots - \frac{y^n}{n!} \right) = e^y\left(\Ein(y) - y + 1\right)-1, \quad y \in \mathbb{R}$$
where $\Ein$ denotes the complementary exponential integral (see ~\cite[6.2.3]{DLMF}, in ~\cite[p. 228]{AS}). We will derive a similar result involving the skew-harmonic numbers.

\begin{theorem} \label{ExpTailTheorem}
    For all $y$, we have
    $$\sum_{n = 1}^\infty \overline{H}_n\left(e^y - 1 - \frac{y}{1!} - \frac{y^2}{2!} - \ldots - \frac{y^n}{n!}\right) = y e^y\left[\Ein(2y) - \Ein(y)\right] - \cosh (y) +1.$$
\end{theorem}

\begin{proof}
    We can easily check that the series evaluates to $0$ for $ y = 0$. For now let us assume $y \neq 0$. We begin by inserting the integral representation of the skew-harmonic numbers \ref{SkewHarmonicInt} and interchange the sum and integral:
    \begin{align*}
        &\sum_{n = 1}^\infty \overline{H}_n\left(e^y - 1 - \frac{y}{1!} - \frac{y^2}{2!} - \ldots - \frac{y^n}{n!}\right)\\
        = &\sum_{n = 1}^\infty \int_0^1 \frac{1 - (-x)^n}{1+x}\left(e^y - 1 - \frac{y}{1!} - \frac{y^2}{2!} - \ldots - \frac{y^n}{n!}\right)\,dx
    \end{align*}
    It is known via the Lagrange error bound that $$\left|e^y - 1 - \frac{y}{1!} - \frac{y^2}{2!} - \ldots - \frac{y^n}{n!} \right| \leq \frac{e^t |y|^{n+1}}{(n+1)!}$$
    for some arbitrary positive $t$ such that $|y| < t$. This inequality was also stated in \cite{ExoticSeries}. Since
    \begin{align*}
        \sum_{n = 0}^\infty \left(e^y - 1 - \frac{y}{1!} - \frac{y^2}{2!} - \ldots - \frac{y^n}{n!} \right)x^n
        = \begin{cases}
        \frac{e^y - e^{xy}}{1-x} & \text{if } x \neq 1\\
        y e^y & \text{if } x = 1
    \end{cases}
    \end{align*}
    (see ~\cite[p. 154]{LimitsSeriesBook}) and
    \begin{align*}
        &\left|\sum_{n=1}^N (1-(-x)^n) \left(e^y - 1 - \frac{y}{1!} - \frac{y^2}{2!} - \ldots - \frac{y^n}{n!}\right)\right|\\
        \leq & \sum_{n=1}^N |1-(-x)^n| \left|e^y - 1 - \frac{y}{1!} - \frac{y^2}{2!} - \ldots - \frac{y^n}{n!}\right|\\
        \leq &2e^t \sum_{n=1}^\infty \frac{|y|^{n+1}}{(n+1)!}\\
        \leq &2e^t(e^{|y|} - 1 - |y|)
    \end{align*}    
    for integers $n \geq 1$, real numbers $x$ such that $0 \leq x \leq 1$, all real numbers $y$ and some arbitrary positive $t$ such that $|y| < t$, we may justify interchanging the sum and integral via Lebesgue's Dominated Convergence Theorem (see \cite[pp. 188-189]{MeasureTheoryBook}). Thus, the series can be rewritten as
    \begin{align*}
        \int_0^1 \left(ye^y - \frac{e^y - e^{-xy}}{1+x}\right)\frac{dx}{1+x}
        =&\int_0^1 \frac{-e^y + y e^y(1+x) + e^{-xy}}{(1+x)^2}\,dx\\
        =& -\frac{e^{y}}{2} + \int_0^1 \frac{y e^y(1+x) + e^{-xy}}{(1+x)^2}\,dx.
    \end{align*}

    \noindent
    Integrating by parts gives us
    \begin{align*}
        -\frac{e^y}{2} + 1 - \frac{e^{-y}}{2} + y \int_0^1 \frac{e^y - e^{-xy}}{1+x}\,dx
        = &1 -\cosh(y) + y \int_0^1 \frac{e^y - e^{-xy}}{1+x}\,dx.
    \end{align*}
    Since $y \neq 0$, we may integrate by substituting $x \mapsto x/y - 1$, which yields
    \begin{align*}
        1 - \cosh (y) +  ye^y\int_y^{2y} \frac{1 - e^{-x}}{x}\,dx
        = & y e^y\left[\Ein(2y) - \Ein(y)\right] - \cosh (y) +1.
    \end{align*}
    We used the fact that $\Ein(z) = \int_0^z \frac{1 - e^{-t}}{t}\,dt$ (see ~\cite[6.2.3]{DLMF}). If we plug in $y = 0$ in the derived expression we also obtain a value of $0$. Thus, the expression is valid for all real $y$.
\end{proof}

\begin{corollary}
    The following equality holds:
    $$\sum_{n = 1}^\infty \frac{\overline{H}_n \{n! e \}}{n!} = e \left[\Ein(2) - \gamma\right] - \cosh(1) - \delta + 1$$
    where $\{x\} = x - \lfloor x \rfloor$ denotes the fractional part of $x$ and $\delta$ is the Euler-Gompertz constant (see ~\cite[p. 423-426]{FinchConstant}).
\end{corollary}

\begin{proof}
    Using the identity (see ~\cite{FractionalPartSeries})
    $$\{n! e\} = n!\left(e - 1 - \frac{1}{1!} - \frac{1}{2!} - \ldots - \frac{1}{n!}\right),$$
    we have
    \begin{align*}
        \sum_{n = 1}^\infty \frac{\overline{H}_n \{n! e \}}{n!}
        &= \sum_{n = 1}^\infty\overline{H}_n\left(e - 1 - \frac{1}{1!} - \frac{1}{2!} - \ldots - \frac{1}{n!}\right).
    \end{align*}
    Using Theorem \ref{ExpTailTheorem} with $y=1$, we simplify to get
    \begin{align*}
        e\left[\Ein(2) - \Ein(1)\right] - \cosh(1) + 1
        &= e\left[\Ein(2) + \mathrm{Ei}(-1) - \gamma\right] - \cosh(1) + 1\\
        &= e\left[\Ein(2) - \gamma\right] + e \mathrm{Ei}(-1) - \cosh(1) + 1
    \end{align*}
    where $\mathrm{Ei}(x) = \dashint_{-\infty}^x \frac{e^t}{t}\,dt$ denotes the exponential integral (see ~\cite[6.2.5]{DLMF}). We used the identity $\mathrm{Ei}(\pm x) = -\Ein(\mp x) + \log(x) + \gamma$ (see ~\cite[6.2.7]{DLMF}). Applying $-e \mathrm{Ei}(-1) = \delta$ (see \cite[p. 303, 424]{FinchConstant}) leads us to the desired result.
\end{proof}

\section{Hardy Series}\label{HardySeriesSect}
\begin{theorem} \label{HardyGenThm}
    For $k \in \{1, 2, 3, \ldots \}$ and $x > 0$, we have
    $$\sum_{n=1}^\infty (-1)^n \left(H_n -\log\sqrt[k]{(n+ x - 1)^{\overline{k}}} - \gamma \right) =  \frac{\gamma}{2} + \frac{1}{k}\log\left[\frac{\Gamma\left(\frac{x+k}{2}\right)}{ \Gamma\left(\frac{x}{2} \right)}\right]$$
    where $z^{\overline{k}} = z(z+1) \ldots (z + k -1)$ denotes the Pochhammer symbol (see ~\cite[5.2(iii)]{DLMF}).
\end{theorem}

\begin{proof}
    We begin by applying Abel's summation formula \eqref{AbelSum} to our summation with $a_n = (-1)^n$ and $b_n = H_n -\log\sqrt[k]{(n+ x - 1)^{\overline{k}}} - \gamma$ to get
    \begin{align*}
        &\lim_{n \to \infty} \frac{-1 + (-1)^n}{2}\left(H_{n+1} -\log\sqrt[k]{(n+ x)^{\overline{k}}} - \gamma \right)\\
        &+ \sum_{n = 1}^\infty \frac{1 - (-1)^n}{2} \left(\frac{1}{n+1} + \log\sqrt[k]{\frac{n + x -1}{n + x + k - 1}} \right).
    \end{align*}
    It is easy to check that the limit evaluates to $0$. This leaves us with the infinite series:
    \begin{align*}
        &\sum_{n = 1, 3, 5,\ldots}  \frac{1 - (-1)^n}{2} \left(\frac{1}{n+1} + \log\sqrt[k]{\frac{n + x - 1}{n + x + k - 1}} \right)\\
        = &\sum_{n = 1}^\infty \left(\frac{1}{2n} + \log\sqrt[k]{\frac{2n + x - 2}{2n + x + k - 2}} \right)\\
        &= \log\sqrt[k]{\prod_{n = 1}^\infty \frac{2n + x - 2}{2n + x + k - 2} e^{k/(2n)}}.
    \end{align*}

    \noindent
    It is known that
    \begin{equation*}
        \Gamma(z+1) = e^{-\gamma z}\prod_{n = 1}^\infty \left(1 + \frac{z}{n}\right)^{-1} e^{z/n}, \quad z \notin \{0, -1, -2, \ldots\},
    \end{equation*}
    which follows from Weierstrass' definition of the gamma function (see ~\cite[p. 176]{Conway}) and the functional equation $z \Gamma(z) = \Gamma(z+1)$ (see ~\cite[p. 256]{AS}). Using this equality, we finally get
    \begin{align*}
        \log\sqrt[k]{\prod_{n = 1}^\infty \frac{2n + x - 2}{2n + x + k - 2} e^\frac{k}{2n}}
        &= \log\sqrt[k]{\frac{e^{-\gamma\left(\frac{x + k - 2}{2}\right)}\prod_{n = 1}^\infty\left(1 + \frac{x + k - 2}{2n}\right)^{-1} e^\frac{x + k - 2}{2n}}{e^{-\gamma\left(\frac{x - 2}{2}\right)}\prod_{n = 1}^\infty\left(1 + \frac{x - 2}{2n}\right)^{-1} e^\frac{x - 2}{2n}}  \cdot e^\frac{k \gamma}{2} }\\
        &= \frac{\gamma}{2} + \frac{1}{k}\log\left[\frac{\Gamma \left(\frac{x + k}{2}\right)}{\Gamma \left(\frac{x}{2}\right)}\right].
    \end{align*}
    
\end{proof}

\begin{corollary}
    \begin{enumerate}[label=(\roman*)]
    The following equalities hold:
    \item \label{HThm1Cor1i} If $k \in \{1 ,2 \ldots\}$, then
    $$\sum_{n=1}^\infty (-1)^n \left(H_n - \log\sqrt[k]{n^{\overline{k}}} - \gamma \right) = \frac{\gamma}{2} - \frac{\log(\pi)}{2k} + \frac{1}{k}\log\Gamma\left(\frac{k+1}{2}\right).$$
    \item \label{HThm1Cor1ii} If $x > 0$, then
    $$\sum_{n=1}^\infty (-1)^n \left(H_n - \log(n + x -1) - \gamma \right) = \frac{\gamma}{2} + \log\left[\frac{\Gamma\left(\frac{x + 1}{2}\right)}{\Gamma \left(\frac{x}{2}\right)} \right].$$
    \end{enumerate}
\end{corollary}

\begin{proof}
    We immediately get \ref{HThm1Cor1i} by letting $x = 1$ in Theorem \ref{HardyGenThm} and using the special value $\Gamma\left(\frac{1}{2}\right) = \sqrt{\pi}$ (see ~\cite[p. 255]{AS}) and simplifying. We immediately get \ref{HThm1Cor1ii} by letting $k = 1$.
\end{proof}

\begin{remark}
    Theorem \ref{HardyGenThm} can be used to solve problems 2.52, 2.53 (a), 2.53 (b), and 2.54 from \cite[pp. 42, 43]{MathAnalysisBook} with the appropriate values of $x$ and $k$.
\end{remark}

\begin{theorem} \label{HardyGenThm2}
For $k \in \{1, 2, \ldots\}$ and $x > 0$, we have
    \begin{align*}
        &\sum_{n = 1}^\infty \left(H_n - \log\sqrt[k]{(n + x - 1)^{\overline{k}}} -\gamma + \frac{x-2}{n} + \frac{k}{2n} \right)\\
        =& \gamma\left( x + \frac{k}{2} - 1\right) + \frac{1 -\log (2 \pi )}{2} + \frac{1}{k}\log\left[\frac{\barnesG(x+k)}{\barnesG(x)}\right]
    \end{align*}
    where $\barnesG$ denotes the Barnes G-function (see ~\cite{OGBarnesG}, ~\cite{BohrMollerupBook}).
\end{theorem}

\begin{proof}
    We know that the Pochhammer symbol $z^{\overline{k}} = z(z+1)\ldots(z+k-1)$ (see ~\cite[p. 256]{AS}, ~\cite[5.2(iii)]{DLMF}). Using this definition, we can turn the summand into a finite sum and change the order of summation:
    \begin{align}
        &\sum_{n = 1}^\infty \left(H_n - \log\sqrt[k]{(n + x - 1)^{\overline{k}}} -\gamma + \frac{x-2}{n} + \frac{k}{2n} \right) \nonumber\\
        = &\frac{1}{k}\sum_{n = 1}^\infty \left(k H_n - \log\left[(n + x - 1)^{\overline{k}}\right] -k\gamma + \frac{k(x-2)}{n} + \frac{k^2}{2n} \right) \nonumber\\
        = &\frac{1}{k}\sum_{j = 1}^k \sum_{n = 1}^\infty \left(H_n - \log(n + x + j -2) - \gamma + \frac{x + j}{n} - \frac{5}{2n} \right). 
    \label{HardyGen2Step1}
    \end{align}

    \noindent
    We now have an infinite series, which we will further simplify. Using the identity
    $$\sum_{n = 1}^\infty\left(H_n - \log(n) - \gamma - \frac{1}{2n} \right) = \frac{\gamma + 1 - \log(2\pi)}{2},$$
    (see ~\cite[p. 145]{LimitsSeriesBook}) we can separate the infinite series in \eqref{HardyGen2Step1} into two convergent series:
    \begin{align*}
        &\sum_{n = 1}^\infty\left(H_n - \log(n) - \gamma - \frac{1}{2n} \right) - \sum_{n = 1}^\infty \left(\log\left(1 + \frac{x + j -2}{n}\right) - \frac{x + j - 2}{n}\right)\\
        = &\frac{\gamma + 1 - \log(2\pi)}{2} - \sum_{n = 1}^\infty \left(\log\left(1 + \frac{x + j -2}{n}\right) - \frac{x + j - 2}{n}\right).
    \end{align*}

    It can be shown by applying the logarithm to Weierstrass' definition of the gamma function (see ~\cite[p. 176]{Conway}) and using the functional equation $\Gamma(z+1) = z\Gamma(z)$ (see ~\cite[p. 256]{AS}) to get that
    $$\log\Gamma(z+1) + \gamma z = - \sum_{n = 1}^\infty \left[\log\left(1 + \frac{z}{n}\right) - \frac{z}{n} \right].$$
    This series can also be found in ~\cite[p. 204]{IrresistibleIntegrals}. Using this identity, we get
    \begin{align*}
        &\gamma\left(x + j - \frac{3}{2}\right) + \frac{1 - \log(2\pi)}{2} + \log\Gamma(x + j - 1).
    \end{align*}
    We now substitute this expression into \eqref{HardyGen2Step1} to obtain the following:
    \begin{align*}
        &\frac{1}{k}\sum_{j = 1}^k \left[\gamma\left(x + j - \frac{3}{2}\right) + \frac{1 - \log(2\pi)}{2} + \log\Gamma(x + j - 1)\right]\\
        =&\gamma\left(x+\frac{k}{2}-1\right) + \frac{1 - \log(2\pi)}{2} + \frac{1}{k}\log\left[\prod_{j = 1}^k\Gamma(x + j - 1)\right].
    \end{align*}
    It can be shown by induction that
    $$\prod_{j = 1}^k\Gamma(x + j - 1) = \frac{\barnesG(x+k)}{\barnesG(x)}$$
    using the functional equation $\barnesG(z+1) = \barnesG(z)\Gamma(z)$ (see ~\cite[5.17.1]{DLMF}).
\end{proof}

\begin{corollary}
    \begin{enumerate}[label=(\roman*)]
    The following equalities hold:
    \item \label{HThm2Cor1i} If $k \in \{1, 2,\ldots\}$, then 
    $$\sum_{n = 1}^\infty \left(H_n - \log\sqrt[k]{n^{\overline{k}}} - \gamma + \frac{k - 2}{2n} \right) = \frac{\gamma k + 1 -\log (2 \pi )}{2} + \frac{\log\barnesG(k+1)}{k}.$$
    \item \label{HThm2Cor1ii} If $x >0$, then
    \begin{align*}
        &\sum_{n=1}^\infty \left(H_n - \log(n + x -1) - \gamma + \frac{x}{n} - \frac{3}{2n}\right)\\
        =& \gamma x + \log\Gamma(x) + \frac{1 - \gamma -\log (2\pi)}{2}.
    \end{align*}
    \end{enumerate}
\end{corollary}

\begin{proof}
    \ref{HThm2Cor1i} immediately follows from setting $x = 1$ in Theorem \ref{HardyGenThm2}. \ref{HThm2Cor1ii} follows from setting $k = 1$ in Theorem \ref{HardyGenThm2}.
\end{proof}

\begin{theorem} \label{HardyGenThm3}
    For $x > 0$, we have
    \begin{align*}
        &\sum_{n = 1}^\infty (-1)^n n\left(H_n - \log(n + x - 1) - \gamma + \frac{x}{n} - \frac{3}{2n}\right)\\
        = &\frac{\gamma }{4} - \frac{x -\log (2) - 1}{2} - 2 \log \left[\frac{\barnesG\left(\frac{x+1}{2}\right)}{\barnesG\left(\frac{x}{2}\right)}\right] + \log\Gamma \left(\frac{x}{2}\right).
    \end{align*}
\end{theorem}

\begin{proof}
    We begin by separating the infinite series:
    \begin{align}
        &\sum_{n = 1}^\infty (-1)^n n\left(H_n - \log(n + x - 1) - \gamma + \frac{x}{n} - \frac{3}{2n}\right) \nonumber\\
        \begin{split}
            = &\sum_{n = 1}^\infty (-1)^n n\left(H_n - \log(n) - \gamma - \frac{1}{2n} \right)\\
            &- \sum_{n = 1}^\infty (-1)^n n\left( \log\left(\frac{n + x - 1}{n}\right) - \frac{x - 1}{n}\right) \label{HardyThm3Step}.
        \end{split}
    \end{align}

    \noindent
    It is known from ~\cite[p. 145]{LimitsSeriesBook} that
    \begin{equation} \label{HardySeriesFurdui}
        \sum_{n = 1}^\infty (-1)^n n\left(H_n - \log(n) - \gamma - \frac{1}{2n} \right) = \frac{\gamma + 1}{4} + \frac{7}{12}\log(2) - 3\log(A)
    \end{equation}
    where $A$ is the Glaisher-Kinkelin constant (see ~\cite[pp. 135-138]{FinchConstant}).

    \noindent
    We are now left with another infinite series which we will evaluate by expressing the summand as an integral and interchanging the order of summation and integration \footnote{The interchange of the sum and integral is justified via Monotone Convergence Theorem (see ~\cite[Theorem 8.5, p. 162]{MeasureTheoryBook}) with $0 \leq \sum_{n = 1}^{2N} \frac{(-1)^{n-1}}{n+t-1} \leq \sum_{n = 1}^{2(N+1)} \frac{(-1)^{n-1}}{n+t-1}$ for real numbers $t >0$ and integers $N \geq 1$.}:
    \begin{align*}
        \sum_{n = 1}^\infty (-1)^n n\left( \log\left(1+\frac{x - 1}{n}\right) - \frac{x -1}{n}\right)
        &= \int_1^x \sum_{n = 1}^\infty \frac{(-1)^{n-1} (t-1)}{n+t-1} \,dt.
    \end{align*}

    \noindent
    From ~\cite[p. 32]{EulerSumContour} we know that
    $$\sum_{n = 0}^\infty \frac{(-1)^n}{n +z} = \frac{1}{2}\psi\left(\frac{z + 1}{2}\right) - \frac{1}{2} \psi\left(\frac{z}{2}\right)$$
    where $\psi$ denotes the digamma function (see ~\cite[pp. 258, 259]{AS}). Plugging in this identity leaves us with an integral, which we will evaluate via integration by parts and expressing the antidervative in terms of the polygamma functions of negative order (see ~\cite{NegaPolygamma}):
    \begin{align*}
        &\frac{1}{2}\int_1^x (t-1)\left(\psi\left( \frac{t + 1}{2} \right) -  \psi\left(\frac{t}{2}\right)\right)\,dt\\
        &= (x-1)\log\left[\frac{\Gamma\left(\frac{x+1}{2}\right)}{\Gamma\left(\frac{x}{2}\right)}\right] - \int_1^x\left(\log\Gamma\left( \frac{t + 1}{2}\right) - \log\Gamma\left(\frac{t}{2}\right) \right)\,dt.
    \end{align*}

    \noindent
    We integrated the digamma function using the identity $\psi(z) = d[\log(\Gamma(z))]/dz$ (see ~\cite[6.3.1, p. 258]{AS}). Since the antidervative of $\log\Gamma(z)$ is $\psi^{(-2)}(z)$ (see ~\cite[p. 196]{NegaPolygamma}), we are left with
    \begin{align*}
        & (x - 1)\log\left[\frac{\Gamma\left(\frac{x+1}{2}\right)}{\Gamma\left(\frac{x}{2}\right)}\right] - 2\left[\psi^{(-2)}\left( \frac{x + 1}{2}\right) - \psi^{(-2)}\left(\frac{x}{2}\right) - \psi^{(-2)}(1) + \psi^{(-2)}\left(\frac{1}{2}\right)\right].
    \end{align*}

    Now we can further simplify by determining the special values of $\psi^{(-2)}\left(\frac{1}{2}\right)$ and $\psi^{(-2)}(1)$. It is known from \cite[p. 196]{NegaPolygamma} that
    \begin{equation}\label{order-2polygamma}
        \psi^{(-2)}(z) = \frac{z(1 -z)}{2} + \frac{z}{2}\log(2\pi) - \zeta'(-1) + \zeta'(-1, z)
    \end{equation}
    where $\zeta'(s, a)$ denotes the partial derivative of the Hurwitz zeta function $\zeta(s,a) = \sum_{n = 0}^\infty (n+a)^{-s}$ with respect to $s$. It is also known from \cite[p. 194]{NegaPolygamma} that
    $$\zeta'(-1) = \frac{1}{12} - \log(A)$$
    and from \cite[p. 205]{MillerDerivatives} that
    $$\zeta'\left(-2k + 1, \frac{1}{2} \right) = - \frac{B_{2k} \log(2)}{4^k k} - \frac{\left(2^{2k - 1} - 1\right)\zeta'(-2k + 1)}{2^{2k - 1}}$$
    where $B_n$ is the $n$th Bernoulli number (see ~\cite[pp. 804-806]{AS}, ~\cite{BernoulliNumbersPaper}) and $k$ is a positive integer.
    From these equalities, we deduce the following:
    $$\psi^{(-2)}\left(\frac{1}{2}\right) = \frac{3 \log (A)}{2}+\frac{5 \log (2)}{24}+\frac{\log (\pi )}{4},\quad \psi^{(-2)}(1) = \frac{\log (2 \pi)}{2}.$$

    \noindent
    With these special values, we now have
    \begin{align}
    \label{AlternateFormHardyGen3}
        \begin{split}
            &\frac{\log (\pi )}{2} - 3 \log(A) + \frac{7}{12}\log (2) + (x -1)\log\left[\frac{\Gamma\left(\frac{x+1}{2}\right)}{\Gamma\left(\frac{x}{2}\right)}\right]\\
            &- 2\left[\psi^{(-2)}\left( \frac{x + 1}{2}\right) -  \psi^{(-2)}\left(\frac{x}{2}\right)\right].
        \end{split}
    \end{align}

    \noindent
    Using \eqref{order-2polygamma} and the identity from \cite[p. 197]{NegaPolygamma}
    $$\zeta'(-1) - \zeta'(-1,z) = \log\barnesG(z+1) - z\log\Gamma(z)$$
    and recombining \eqref{HardySeriesFurdui} and \eqref{AlternateFormHardyGen3} with \eqref{HardyThm3Step}, we get the desired result.
\end{proof}

\begin{remark}
    We could also have expressed the closed form of the series from Theorem \ref{HardyGenThm3} with \eqref{AlternateFormHardyGen3}. We can then use \eqref{order-2polygamma} and the formulae from \cite{MillerDerivatives} and \cite{PolygammaSpecialValues} to determine special values. For example, if it were necessary to evaluate $\psi^{(-2)}(p)$ at $p = \frac{1}{4}, \frac{3}{4}$, one may use the following formulae from those papers:
    \begin{align*}
        \left.
        \begin{array}{cc}
                \zeta'\left(-2k+1, \frac{1}{4}\right) & \\
                \zeta'\left(-2k+1, \frac{3}{4}\right) &
    	\end{array}
        \right\}
        =& \mp \frac{\left(4^k - 1\right)B_{2k}\pi}{4^{k+1}k} + \frac{\left(4^{k-1} - 1\right)B_{2k} \log(2)}{2^{4k-1}k}\\
        &\mp \frac{(-1)^k \psi^{(2k-1)}\left(\frac{1}{4}\right)}{4(8\pi)^{2k-1}} - \frac{\left(2^{2k-1} - 1\right)\zeta'(-2k+1)}{2^{4k-1}}
    \end{align*}
    and
     \begin{align*}
        \left.
        \begin{array}{cc}
                \psi^{(2k-1)}\left(\frac{1}{4}\right) & \\
                \psi^{(2k-1)}\left(\frac{3}{4}\right) &
    	\end{array}
        \right\}
        =& \frac{4^{2k-1}}{2k}\left[\pi^{2k}\left(2^{2k} -1 \right)\left|B_{2k} \right| \pm 2(2k)!\beta(2k)\right]
    \end{align*}
    where $\psi^{(n)}$ denotes the polygamma function of $n$th order (see ~\cite[p. 260]{AS}, ~\cite[5.15]{DLMF}) and $\beta(s) = \sum_{n = 0}^\infty \frac{(-1)^n}{(2n+1)^s}$ denotes the Dirichlet beta function (see ~\cite[pp. 807, 808]{AS}).
\end{remark}

By letting $x$ in Theorem \ref{HardyGenThm3} take on certain values, we may derive some exotic series.

\begin{corollary}\label{HThm3Cor}
    \begin{enumerate}[label=(\roman*)]
    The following equalities hold:
    \item \label{HThm3Cor1i} If $x = 1/2$, we have
    $$\sum_{n = 1}^\infty (-1)^n n \left(H_n - \log\left( n - \frac{1}{2}\right) - \gamma - \frac{1}{n} \right) = \frac{\gamma + 1}{4} -\frac{G}{\pi } - \frac{1}{4}\log \left(\frac{\varpi^2}{\pi}\right)$$
    where $G = \beta(2) = \sum_{n=0}^\infty \frac{(-1)^n}{(2n+1)^2}$ is Catalan's constant (see ~\cite[pp. 53-57]{FinchConstant}) and $\varpi$ is the lemniscate constant (see ~\cite[pp. 420-422]{FinchConstant}).
    
    \item \label{HThm3Cor1ii} If $x = 2/3$, then
    \begin{align*}
        &\sum_{n = 1}^\infty (-1)^n n \left(H_n - \log\left( n - \frac{1}{3}\right) - \gamma - \frac{5}{6n} \right)\\
        &= \frac{\gamma + 1}{4} - \frac{5\kappa}{6 \pi} - \log(A) + \frac{19}{36}\log(2) + \frac{\log(3)}{24} + \frac{1}{3}\log\left[\frac{\Gamma\left(\frac{5}{6}\right)}{\Gamma\left(\frac{1}{3}\right)} \right].
    \end{align*}
    where $\kappa$ denotes Gieseking's constant (see \cite{GiesekingOEIS}).
    \end{enumerate}

\end{corollary}

\begin{proof}
    If we let $x=1/2$ in Theorem \ref{HardyGenThm3} and use the special value
    $$\log\barnesG\left(\frac{3}{4}\right) = \log\barnesG\left(\frac{1}{4}\right) + \frac{G}{2\pi} - \frac{\log(2)}{8} - \frac{\log(\pi)}{4} + \log\Gamma\left(\frac{1}{4}\right)$$
    and the identity
    $$\varpi = \frac{\Gamma^2\left(\frac{1}{4}\right)}{2\sqrt{2\pi}}$$
    (see ~\cite[p. 94]{BarnesZetaSeries}, ~\cite[p. 420]{FinchConstant}), we get the desired result in \ref{HThm3Cor1i}.

    If we let $x = 2/3$, we get
    \begin{equation} \label{FinalCorollary}
        \frac{\gamma}{4}+\frac{\log (2)}{2} +\frac{1}{6} + 2\log \barnesG\left(\frac{1}{3}\right) -2 \log \barnesG\left(\frac{5}{6}\right) +\log\Gamma \left(\frac{1}{3}\right).
    \end{equation}

    \noindent
    We will make use of the following special values (see \cite{BarnesGAdam}):
    \begin{align} \label{Step1iFinalCorollary}
        \log \barnesG\left(\frac{1}{3}\right) &= \frac{\log(3)}{72} + \frac{\pi}{18 \sqrt{3}} - \frac{2}{3}\log\Gamma\left(\frac{1}{3}\right) - \frac{4}{3}\log(A) - \frac{\psi^{(1)}\left(\frac{1}{3}\right)}{12 \pi \sqrt{3}} + \frac{1}{9}\\
        \label{Step1iiFinalCorollary}
        \log \barnesG\left(\frac{5}{6}\right) &= -\frac{\log(12)}{144} + \frac{\pi}{20\sqrt{3}} - \frac{1}{6}\log\Gamma\left(\frac{5}{6}\right) - \frac{5}{6}\log(A) - \frac{\psi^{(1)}\left(\frac{5}{6}\right)}{40 \pi \sqrt{3}} +\frac{5}{72}.
    \end{align}

    \noindent
    Using the polygamma multiplication formula (see ~\cite[5.15.7]{DLMF})
    $$\psi^{(n)}(mz) = \frac{1}{m^{n+1}}\sum_{k = 0}^{m-1}\psi^{(n)}\left(z + \frac{k}{m} \right), \quad m,n \in \{1, 2, \ldots\}$$
    with $m=2$, $n = 1$, and $z = 1/3$ alongside the polygamma reflection formula (see ~\cite[5.15.6]{DLMF})
    \begin{equation}
    \label{polyreflect}
        \psi^{(n)}(1-z) + (-1)^{n-1} \psi^{(n)}(z) = (-1)^n \pi \frac{d^n}{dz^n}\cot(\pi z)
    \end{equation}
    with $z = 2/3$ and $n = 1$,
    we get that
    \begin{equation}
        \label{polygamma56}
        \psi^{(1)}\left(\frac{5}{6}\right) = \frac{16 \pi ^2}{3} - 5\psi^{(1)}\left(\frac{1}{3}\right).
    \end{equation}
    Substituting \eqref{polygamma56} into \eqref{Step1iiFinalCorollary} and plugging \eqref{Step1iFinalCorollary} and \eqref{Step1iiFinalCorollary} into \eqref{FinalCorollary}, we have
    \begin{equation} \label{FinalCorollary2}
        \frac{\gamma + 1}{4} + \frac{5 \pi}{18 \sqrt{3}} - \log(A) + \frac{19}{36}\log(2) + \frac{\log(3)}{24} + \frac{1}{3}\log\left[\frac{\Gamma\left(\frac{5}{6}\right)}{\Gamma\left(\frac{1}{3}\right)} \right] - \frac{5 \psi^{(1)}\left(\frac{1}{3}\right)}{12 \pi \sqrt{3}}.
    \end{equation}
    We use the identity
    $$\kappa = \frac{9 - \psi^{(1)}\left(\frac{2}{3}\right) + \psi^{(1)}\left(\frac{4}{3}\right)}{4 \sqrt{3}}$$
    and substitute the $\psi^{(1)}(2/3) = 4\pi^2/3 - \psi^{(1)}(1/3)$, which comes from \eqref{polyreflect}, and $\psi^{(1)}(4/3) = \psi^{(1)}(1/3) - 9$, which comes from the recursive relationship
    $$\psi^{(n)}(z+1) = \psi^{(n)}(z) + (-1)^n n! z^{-n-1}$$
    with $n = 1$ and $z = 1/3$ (see ~\cite[5.15.5]{DLMF}). This gives us the special value
    $$\psi^{(1)}\left(\frac{1}{3}\right) = 2 \sqrt{3} \kappa + \frac{2\pi^2}{3}.$$
    Substituting this value into \eqref{FinalCorollary2} yields the result in \ref{HThm3Cor1ii}.
\end{proof}

\section{Further Research}
The interested reader may consider a generalization to the series $S_2$:
$$\sum_{n=1}^\infty \frac{\overline{H}_n}{n}\left(\zeta(k) - 1 - \frac{1}{2^k} - \ldots \frac{1}{n^k}\right),\quad n \in \{2, 3, \ldots\}.$$
One may also consider other analogues to the Hardy series. For example, one may study the series
$$\sum_{n=1}^\infty (-1)^n n^k(H_n - \log(n) - \gamma - \sigma_k(n)), \quad k\in\{0,1,\ldots\}$$
where $\sigma_k(n)$ is the correctional term that ensures the convergence of the series.



\begin{thebibliography}{20}

\bibitem{AS}
    {\sc M. Abramowitz, I. A. Stegun (Eds)}, 
    {\it Handbook of Mathematical Functions with Formulas, Graphs, and Mathematical Tables}, 
    National Bureau of Standards, Applied Mathematics Series \textbf{55}, 
    9th printing, Washington, 1964.

\bibitem{BarnesGAdam}
    {\sc V. Adamchik},
    {\it On the Barnes function},
    Proc. Int. Symp. Symb. Algebr. Comput. ISSAC, (2001) 15-20.

\bibitem{NegaPolygamma}
    {\sc V. Adamchik},
    {\it Polygamma functions of negative order},
    J. Comput. Appl. Math \textbf{100(2)}, (1998) 191–199.

\bibitem{BaileyEuler}
    {\sc D.H. Bailey, J.M. Borwein, R. Girgensohn},
    {\it Experimental evaluation of Euler sums},
    Exp. Math. \textbf{3},
    (1994) 17-30.

\bibitem{OGBarnesG}
    {\sc E.W. Barnes},
    {\it The theory of the G-function},
    Quart. J. Pure Appl. Math. \textbf{31}, (1900) 264-314.

\bibitem{IrresistibleIntegrals}
    {\sc G. Boros, V. Moll},
    {\it Irresistible integrals: symbolics, analysis and experiments in the evaluation of integrals},
    Cambridge University Press, 2004.

\bibitem{BorweinEuler}
    {\sc D. Borwein, J.M. Borwein, R. Girgensohn},
    {\it Explicit evaluation of Euler sums},
    Proc. Edinburgh Math. Soc. \textbf{38},
    (1995) 277-294.

\bibitem{BomanRogersAnalysis}
    {\sc E. Boman, R. Rogers},
    {\it How we got from there to here: A story of real analysis}
    Open SUNY Textbooks, 2014.

\bibitem{ExoticSeries}
    {\sc K.N. Boyadzhiev},
    {\it Exotic series with Bernoulli, harmonic, Catalan, and Stirling numbers},
    arXiv, \url{https://arxiv.org/abs/2110.00689}, 2021.

\bibitem{SeriesBook}
    {\sc T.J.I’A. Bromwich},
    {\it An introduction to the theory of infinite series},
    Third Edition, American Mathematical Society, Providence, Rhode Island, 1991.

\bibitem{BarnesZetaSeries}
    {\sc J. Choi, H.M. Srivastava},
    {\it Certain classes of series involving the zeta function},
    J. Math. Anal. \textbf{231(1)}, (1999) 91-117.

\bibitem{MiscSeriesZeta}
    {\sc J. Choi, H.M. Srivastava, J.R. Quine},
    {\it Some series involving the zeta function},
    Bull. Austral. Math. Soc. \textbf{51}, (1995) 383-393.
        
\bibitem{Conway}
    {\sc J.B. Conway},
    {\it Functions of one complex variable I},
    Functions of One Complex Variable,
    Springer, 1978.

\bibitem{FinchConstant}
    {\sc S. Finch},
    {\it Mathematical constants},
    Cambridge University Press, 2003.

\bibitem{EulerSumContour}
    {\sc P. Flajolet, B. Salvy},
    {\it Euler sums and contour integral representations},
    Exp. Math. \textbf{7(1)}, (1998) 15-35.

\bibitem{LimitsSeriesBook}
    {\sc O. Furdui},
    {\it Limits, series, and fractional part integrals}
    Problems in Mathematical Analysis,
    Springer, New York, 2013.

\bibitem{FractionalPartSeries}
    {\sc O. Furdui, A. Sintamarian},
    {\it Exotic series with fractional part function}
    Gazeta Matematica, Seria A. \textbf{XXXV}, (2017, 10) 1-12.

\bibitem{BernoulliNumbersPaper}
    {\sc H. W. Gould},
    {\it Explicit formulas for Bernoulli numbers},
    Am Math Mon. \textbf{79(1)}, (1972) 44–51.
    
\bibitem{PolygammaSpecialValues}
    {\sc K. Kölbig},
    {\it The polygamma function $\psi^{(k)}(x)$ for $x = 1/4$ and $x = 3/4$},
    J. Comput. Appl. Math. \textbf{75}, (1996) 43-46.

\bibitem{PolyLogMisc}
    {\sc L. Lewin},
    {\it Polylogarithms and associated functions},
    North Holland, 1981.

\bibitem{BohrMollerupBook}
    {\sc J.L. Marichal, N. Zenaïdi},
    {\it A generalization of Bohr-Mollerup’s theorem for higher order convex functions},
    Springer, Cham, Switzerland, 2022.

\bibitem{MillerDerivatives}
    {\sc J. Miller, V. Admamchik},
    {\it Derivatives of the Hurwitz zeta function for rational arguments},
    J. Comput. Appl. Math \textbf{100(2)}, (1998) 201–206.

\bibitem{NielsEuler}
    {\sc N. Nielsen},
    {\it Die Gammafunktion},
    American Mathematical Soc., 2005.

\bibitem{GiesekingOEIS}
    {\sc OEIS Foundation Inc. (2023)},
    {\it Decimal expansion of Gieseking's constant}, 
    Entry A143298 in The On-Line Encyclopedia of Integer Sequences,
    \url{https://oeis.org/A143298}.

\bibitem{HarmonicSeriesLogIntBook}
    {\sc A.S. Olaikhan},
    {\it An Introduction to the harmonic series and logarithmic integrals: for high school students up to researchers}
    Second Edition, Ali Shadhar Olaikhan, Phoenix, AZ, 2021.

\bibitem{DLMF}
    {\sc F. W. J. Olver, A. B. Olde Daalhuis, D. W. Lozier, B. I. Schneider, R. F. Boisvert, C. W. Clark, B. R. Miller, B. V. Saunders, H. S. Cohl, and M. A. McClain, eds.},
    {\it NIST digital library of mathematical functions},
    \url{http://dlmf.nist.gov/}, Release 1.1.8 of 2022-12-15.
    

\bibitem{SandiferEuler}
    {\sc C.E. Sandifer},
    {\it How Euler did it},
    MAA Spectrum, Mathematical Association of America, 2007.

\bibitem{MathAnalysisBook}
    {\sc A. Sîntămărian, O. Furdui},  
    {\it Sharpening mathematical analysis skills}, 
    Springer, 2021.

\bibitem{SrivastavaEuler}
    {\sc H. M. Srivastava, J. Choi},
    {\it Series associated with the zeta and related functions},
    Springer, Dordrecht, 2001.

\bibitem{AlmostImpossibleBook1}
    {\sc C. I. Vălean},
    {\it (Almost) impossible integrals, sums, and series},
    Springer, 2019.

\bibitem{ValeanSkewHarmonicPaper}
    {\sc C. I. Vălean},
    {\it The calculation of a harmonic series with a weight 5 structure, involving the product of harmonic numbers, $H_n H_n^{(2)}$},
    \url{https://www.researchgate.net/publication/336378340}, 2019.

\bibitem{MeasureTheoryBook}
    {\sc J. Yeh},
    {\it Real analysis: theory of measure and integration},
    Third Edition, World Scientific Publishing Company, 2014.

\end{thebibliography}
\end{document}